\documentclass[12pt,a4paper]{article}
\usepackage[latin1]{inputenc}
\usepackage{amsmath}
\usepackage{amsfonts}
\usepackage{amssymb}
\usepackage{amsthm}
\usepackage{changepage}
\usepackage{color}
\usepackage{graphicx} 
\usepackage[normalem]{ulem}

\newtheorem{theo}{Theorem}[section]
\newtheorem{lem}[theo]{Lemma}
\newtheorem{cor}[theo]{Corollary}
\newtheorem{prop}[theo]{Proposition}

\theoremstyle{definition}
\newtheorem{defi}[theo]{Definition}
\newtheorem{exa}[theo]{Example}

\newcommand{\e}[1]{{\rm e}^{#1}}
\newcommand{\N}{\mathbb N}
\newcommand{\R}{\mathbb R}

\newcommand{\C}{\mathbb C}

\begin{document}

\begin{center}
\section*{Uniqueness of local, analytic solutions to singular ODEs}

\large{Thomas de Jong\\}
\small{Media Analytics and Computing laboratory, School of Information Science and Engineering, \\
Xiamen University,
Xiamen 361005, China.\\
{\tt t.g.de.jong.math@gmail.com}}
\\[2mm]

\large{Patrick van Meurs\\}
\small{Faculty of Mathematics and Physics, 
Institute of Science and Engineering, \\
Kanazawa University,
Kanazawa, Japan.\\
{\tt pjpvmeurs@staff.kanazawa-u.ac.jp}}
\\[2mm]

\end{center}

\begin{small} \noindent \textbf{Abstract.} We study local, analytic solutions for a class of initial value problems for singular ODEs. We prove existence and uniqueness of such solutions under a certain non-resonance condition. Our proof translates the singular initial value problem to an equilibrium problem of a regular ODE. Then, we apply classical invariant manifold theory. We demonstrate that the class of ODEs under consideration captures models which describe the shape of axially symmetric surfaces which are closed on one side. Our main result guarantees smoothness at the tip of the surface. 
\vspace{3mm}

\noindent {\bf Keywords:} Singularities, ordinary differential equations, asymptotics, analyticity, invariant manifolds. \\ 
\end{small}

\section{Introduction \label{sec:intro}}

The problem in this paper is motivated by a specific application in which the shape of an axial symmetric surfaces with a smooth tip is sought. We postpone the details of this model to Section \ref{sec:application}, and focus here on a more general setting.

Suppose we are modeling an axially symmetric surface with a smooth tip in cylindrical variables as the solution of an ordinary differential equation. We parametrize the axial co-ordinate $z$ with respect to the axial distance co-ordinate $r$; see Figure \ref{fig:axial}. Hence, $r$ is the independent variable in the ODE and $z$ is the dependent variable. We are interested in  deriving conditions under which solutions $z(r)$ are unique. The requirement that the tip of the surface is smooth translates to the requirement that $z$  is even and smooth in a neighborhood around $0$. For convenience, we further assume that $z$ is locally analytic. Then, the requirements on $z$ can be reformulated as the requirement that there exist $\varepsilon > 0$ and $g \in C^{\omega}((-\varepsilon,\varepsilon), \mathbb{R})$ with $g(0)=0$ such that
\begin{align}
z(r) = g(r^2)
\quad \text{for all } |r|<\varepsilon. \label{eq:intro_an}
\end{align}

We assume that the governing equations are of the form
\begin{align}
\frac{dz}{dr} = \frac{V(z,r^2)}{r}, \qquad z(0) = 0, \label{eq:intro}
\end{align} 
where $ 0<r < r_0$ and $V \in C^{\omega}(\mathbb{R}^2, \mathbb{R})$ with $V(0) = 0$. The singularity $1/r$ arises from the expression of the gradient in cylindrical coordinates. The argument $r^2$ in $V$ forces evenness of the solution. In applications the argument $r^2$ arises naturally from surface force terms or from a cumulative flux. We are interested in finding sufficient conditions on $V$ for which solutions of \eqref{eq:intro} satisfying \eqref{eq:intro_an} are unique. 

\begin{figure}[h]
\centering
\includegraphics[width=5cm]{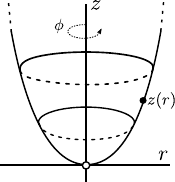}
\caption{ Axial symmetric surface with smooth tip: The surface is described in cylindrical variables by $(\phi,r,z(r))$. The axial co-ordinate, $z$, is parametrised with respect to the axial co-ordinate $r$.  \label{fig:axial}}
\end{figure}

\begin{adjustwidth}{2.5em}{2.5em}
\begin{exa} \label{ex} 
A simple but insightful example of \eqref{eq:intro} is when $V$ is linear. Then, \eqref{eq:intro} reads as
\begin{align}
z' = \lambda \frac{z}{r} +  b r \label{eq:example}
\end{align}
with $\lambda,b \in \R$. The general solution is given by 
\begin{align*}
z(r)  = C_1 r^{\lambda} + \tilde{z}(r),
\end{align*}
where $C_1$ a free constant and 
\begin{align*}
\tilde{z}(r) :=
\begin{cases}
\frac{b}{2-\lambda} r^2  & {\rm if} \;   \lambda  \neq 2, \\
 b r^2 \log(r)  & {\rm if} \;   \lambda  = 2  .
\end{cases}
\end{align*}
If $\lambda \neq 2$, then the solution $z$ with $C_1 = 0$ satisfies \eqref{eq:intro_an}. If $\lambda \notin 2 \mathbb{N}_+$, then this is the only solution which satisfies \eqref{eq:intro_an}. \\
\end{exa}
\end{adjustwidth}

The condition $\lambda \notin 2 \mathbb{N}_+$ in Example \ref{ex} is a type of non-resonance condition. It suggests that \eqref{eq:intro} will not have unique solutions for any analytic $V$, and that in addition a requirement such as $\partial_z V(0) \notin 2 \mathbb{N}_+$ is needed. This additional requirement turns out to be exactly the sufficient condition on $V$ in the main result in this paper.

The biological application in Section \ref{sec:application} requires the following generalization of \eqref{eq:intro_an}--\eqref{eq:intro} to higher dimensions. The solution concept for the unknown $x \in \R^n $ is that there exists some $g \in C^{\omega}((-\varepsilon,\varepsilon), \R^n)$ with $g(0)=0$ such that
\begin{align}
x(r) = g(r^2), \; {\rm for\; all \;} |r|<\varepsilon. \label{eq:intro_an_n}
\end{align}
The ODE for $x$ is
\begin{align}
\frac{dx}{dr} = \frac{V(x,r^2)}{r}, \qquad x(0) =0, \label{eq:intro_n}
\end{align} 
where $V \in C^{\omega}(\mathbb{R}^{n+1}, \mathbb{R}^n)$ with $V(0) = 0$. To reveal the connection with Example \ref{ex}, we expand
\begin{align}
V(x,r^2)= A x + b r^2 + f(x,r^2), \label{eq:V}
\end{align}
where $A \in \mathbb{R}^{n \times n}$, $b \in \R^n$ and $f(y) = O(|y|^2)$ with $y = (x,r^2)$. Our main result, Corollary \ref{cor:main}, is that a sufficient condition on $V$ for the existence and uniqueness of solutions to \eqref{eq:intro_n} satisfying \eqref{eq:intro_an_n} is that $\lambda_i \notin 2 \mathbb{N}_+$ for all $1 \leq i \leq n$, where $\lambda_i$ are the eigenvalues of $A$.

To prove Corollary \ref{cor:main}, we transform \eqref{eq:intro_n} into an autonomous system with no singularity and where the problem of uniqueness of analytic solutions turns into an equilibrium study. To remove the singularity, we introduce the independent variable $t$ given by $r=\e{t}$, and obtain from \eqref{eq:intro_n} that $\hat x(t) := x(e^t)$ satisfies
\begin{align}
 \dot{\hat x} := \frac{d\hat x}{dt} = V(\hat x, e^{2t}). \label{eq:nonaut}
\end{align}
Note that the initial condition in \eqref{eq:intro_n} at $r = 0$ is transformed to the equilibrium point $0$ at $t = -\infty$. To make this system autonomous, we introduce the dependent variable $\rho = e^{2t}$ and consider
\begin{equation} \label{eq:intro_DS}
  \left\{ \begin{aligned}
    \dot{\hat x} &= V(\hat x, \rho) \\
    \dot \rho &= 2\rho.
  \end{aligned} \right.
\end{equation}
Along this transformation, \eqref{eq:intro_an_n} implies that
\begin{align} \label{eq:intro_DS_an}
\hat{x}(t) = {g}(\rho(t)) \quad \text{for all } t \text{ with } 0 < \rho(t)< \hat{\varepsilon}.
\end{align}
Hence, Corollary \ref{cor:main} can be formulated in terms of local properties of the equilibrium of \eqref{eq:intro_DS}. Theorem \ref{theo:main} provides the precise statement. We consider Theorem \ref{theo:main} as our main mathematical result, and Corollary \ref{cor:main} as the main statement regarding its application. We prove Theorem \ref{theo:main} by using classical  invariant manifold theory \cite{abbondandolo2006global, carr2012applications, guckenheimer2013nonlinear, irwin1970stable, irwin2001smooth, kelley1966stable, shub1978stabilite}.

To demonstrate the applicability and use of Corollary \ref{cor:main}, we apply it to the Ballistic Ageing Thin viscous Sheet (BATS) model \cite{jong2020modelling}. The BATS model describes tip growth for single fungal  cells in terms of a system of ODEs for an axial symmetric surface with a smooth tip. We show that Corollary \ref{cor:main} provides sufficient conditions for the parameters in the BATS model under which unique solutions with a smooth tip exist. More specifically, our result implies that if the expansion is of sufficiently high order then it approximates the smooth solution at the tip. This gives a theoretical motivation for the numerical approach in \cite{jong2019numerics} in which approximations to solutions to the BATS model are constructed from asymptotic expansions. 

The paper is organized as follows. In Section \ref{sec:main} we formulate \eqref{eq:intro_DS}--\eqref{eq:intro_DS_an} in a general dynamical systems framework and present Theorem \ref{theo:main}. We prove it in Section \ref{sec:proof}. In Section \ref{sec:application} we formulate and prove Corollary \ref{cor:main} and apply it to the BATS model. In Section \ref{sec:conc} we give concluding remarks and suggest future research. 

\section{Main mathematical result \label{sec:main}}

In this section we present Theorem \ref{theo:main}, which is our main mathematical result. 

We define the phase space
\begin{align*}
M: = \{ y = (x, \rho) \in \mathbb{R}^n \times \mathbb{R} \}.
\end{align*}
On $M$ we consider the generalization of \eqref{eq:intro_DS} given by
\begin{gather}
\left\{ \begin{aligned}
\dot x &= A x  + b \rho +   f(x,\rho) \\
\dot \rho &= \sigma \rho ,
\end{aligned} \right. \label{eq:gov_n}
\end{gather}
where  $b \in \mathbb{R}^{n}$, $\sigma > 0$, $A \in \R^{n \times n}$ and $f \in C^{\omega}(\mathbb{R}^{n+1}, \mathbb{R}^n)$ with $f(y) = O(|y|^2)$ as $y \to 0$.  The vector field corresponding to \eqref{eq:gov_n} has an equilibrium at $0$ with linearization
\begin{align}
\begin{bmatrix}
A & b \\
0_{1 \times n} & 2 
\end{bmatrix}. \label{eq:lin}
\end{align}
%

\begin{defi}[$\rho$-analytic] A solution $(x,\rho)$ of \eqref{eq:gov_n} is called $\rho$-analytic if there exist $R, \varepsilon >0$ and $g \in C^{\omega}((-\varepsilon, \varepsilon), \mathbb{R}^n)$ with $g(0)=0$ such that $x= g(\rho)$ on $(-\infty, -R)$.
\label{defi:rho-anal}
\end{defi}

We make three preliminary observations. First, the equilibrium $(x, \rho) = 0$ is a $\rho$-analytic solution. Second, since $\rho$ can be solved directly from \eqref{eq:gov_n} (i.e.\ $\rho(t) = c_0 \e{\sigma t}$ for some $c_0 \in \R$) the $x$-component of a $\rho$-analytic solution can be expressed as an analytic function of $\e{\sigma t}$. Third, the freedom in the choice of $c_0$ corresponds to a translation in time. Hence, $\rho$-analytic solutions are invariant in translation in time, and it suffices to consider $c_0 \in \{-1,0,1\}$. 

We want to express Definition \ref{defi:rho-anal} in the language of invariant manifolds. We introduce $W^{\rho}(0) \subset M$ as the set of all initial conditions for which the corresponding solution is $\rho$-analytic. Denote by $W^{u}(0)$ the unstable manifold corresponding to \eqref{eq:gov_n}.

\begin{prop} \label{propo} $W^{\rho}(0) \subset W^{u}(0)$.
\end{prop}

\begin{proof} By the center manifold theorem \cite{guckenheimer2013nonlinear} $W^{\rho}(0)$ is contained in either $W^{u}(0)$, $W^{c}(0)$ or $W^{s}(0)$, which are the unstable, center and stable manifolds, respectively. By Definition \ref{defi:rho-anal} it follows that $W^{\rho}(0) \not \subset W^{s}(0)$. 

It is left to show that $W^{\rho}(0) \not \subset W^{c}(0)$. Denote by $E^{u}$ and $E^{c}$ the linear eigenspaces corresponding to the stable and center subspace, respectively. Let $(v, r) \in M$ be a nonzero tangent vector of $W^{\rho}(0)$ at $0$. Definition \ref{defi:rho-anal} implies $r \neq 0$.  It follows from \eqref{eq:lin} that $(\tilde{v}, \tilde r) \in E^c$ implies $\tilde r = 0$. Hence, $(v, r) \notin E^c$ and thus $W^{\rho}(0) \not \subset W^{c}(0)$. 
\end{proof}


\begin{theo} Let $\lambda_1, \ldots, \lambda_n$ be the eigenvalues of $A$. If  $\rm{Re} (\lambda_i) \notin \sigma \N_+$ for all $1 \leq i \leq n$, then $W^{\rho}(0)$ is a one-dimensional smooth manifold and there exist $\varepsilon > 0$ and $g \in C^{\omega}( (-\varepsilon, \varepsilon) , \mathbb{R}^n)$ such that
\begin{align}
\{ (x,\rho) \in W^{\rho}(0) \; : \;   \rho \in (-\varepsilon, \varepsilon) \} = \{ (g(\rho),\rho) \; : \;   \rho \in (-\varepsilon, \varepsilon) \}. \label{eq:mani_anal}
\end{align}

 \label{theo:main}
\end{theo}

Equation \eqref{eq:mani_anal} implies that $W^{\rho}(0)$ is the union of three $\rho$-analytic solution trajectories: the equilibrium $0$, $W^{\rho}(0)$ restricted to $\rho>0$ and $W^{\rho}(0)$ restricted to $\rho<0$.  
Furthermore, there exist $\rho$-analytic solution, and they are unique modulo translation in time if we include the constraint $\rho=0$, $\rho>0$ or $\rho<0$. Note that orbits on $W^{\rho}(0)$ cannot have a complicated geometry in $M$ since \eqref{eq:gov_n} is linear in $\rho$.

\section{Proof Main Theorem \label{sec:proof}}

The proof of the main theorem, Theorem \ref{theo:main}, is given at the end of this section. It relies on Lemmas \ref{lem:negeig} and \ref{lem:trans}. Lemma \ref{lem:negeig} states that Theorem \ref{theo:main} holds under the additional assumption that ${\rm Re}(\lambda_i)<0$. The proof of Lemma \ref{lem:negeig} relies on the analytic version of the unstable manifold theorem. This gives analyticity of the solution without the need to prove convergence of power series. Lemma \ref{lem:trans} introduces a recursive transformation under which the eigenvalues can be shifted to the left half-plane in $\C$ such that Lemma \ref{lem:negeig} can be applied. At each iteration of this transformation we linearize around the next coefficient in the power series of the analytic solution.

\begin{lem} If  ${\rm Re}(\lambda_i)<0$ for all $1 \leq i \leq n$, then $W^{\rho}(0)$ is a one-dimensional smooth manifold satisfying \eqref{eq:mani_anal}. \label{lem:negeig} 
\end{lem}


%
%
%

\begin{proof} As preparation, we denote by $E^s$ and $E^u$ the stable and unstable subspace of the eigenspaces of \eqref{eq:lin}, respectively. Since ${\rm Re}(\lambda_i)<0$, we observe that ${\rm dim}(E^s)= n$, ${\rm dim}(E^u)= 1$ and that the eigenvalue corresponding to $E^u$ is $2$. As in the proof of Proposition \ref{propo} we obtain that $E^u = \langle  \overline{v} \rangle$ with $\overline{v} := (v,1) \in M$ for some $v \in \R^n$. 

First, we prove Lemma \ref{lem:negeig} for $W^u(0)$ instead of $W^\rho(0)$. We start with property \eqref{eq:mani_anal}. With this aim, we prepare for applying the local unstable manifold theorem \cite{abbondandolo2006lectures}. We use the corresponding notation.  Let $E^u(\varepsilon) := E^u \cap B_{\varepsilon}(0)$ and $E^{s}(\varepsilon)  := E^s \cap B_{\varepsilon}(0)$, where $B_{\varepsilon}(0)$ is the ball in $M$ centred at $0$ with radius $\varepsilon > 0$. Denote by  $W^u_{\rm loc, \varepsilon}(0)$ the local unstable manifold induced by $B_{\varepsilon}(0)$. Then, the local unstable manifold theorem states that $W^u_{\rm loc, \varepsilon}(0)$ is the graph of some $\hat{g} \in C^{\omega}(E^{u}(\varepsilon),E^{s}(\varepsilon) )$ with $\hat{g}(0)=0$ and $D\hat{g}(0)=0$. Using this and recalling $E^u = \langle  \overline{v} \rangle$, we parametrize $W^u_{\rm loc, \varepsilon}(0)$ by $\overline{g} \in C^\omega((-\varepsilon, \varepsilon), \mathbb{R}^{n+1})$ given by  $\overline{g}(\rho) = \hat{g}(\rho \overline{v}) +  \rho \overline{v}$. Restricting $\overline{g}$ to the $x$-component we obtain that $W^u_{\rm loc, \varepsilon}(0)$ satisfies \eqref{eq:mani_anal}. 

Next we extend $W^u_{\rm loc, \varepsilon}(0)$ to the global manifold
\begin{align*}
W^{u}(0) = \bigcup_{t \geq 0 } \phi_t(W^u_{\rm loc, \varepsilon}(0)),
\end{align*}
with $\phi$ denoting the flow of \eqref{eq:gov_n}. By this construction, $W^{u}(0)$ is a one-dimensional smooth manifold and $W^{u}(0) \subset W^{\rho} (0)$. Hence, $W^{u}(0)$ satisfies Lemma \ref{lem:negeig}, and, by Proposition \ref{propo}, $W^{\rho}(0) = W^{u}(0)$. This completes the proof. 
\end{proof}

In Lemma \ref{lem:trans} we recursively expand $x$ in terms of powers of $\rho$. In more detail, we consider the governing equations corresponding to $\tilde{x}$ induced from 
\begin{align}
x(t)=\rho(t) (\tilde{x}(t)+\tilde{c}) \label{eq:xtilde} 
\end{align}
for a certain $\tilde c \in \R^n$. 

\begin{lem} Let $\lambda_i \neq \sigma$ for all $1 \leq i \leq n$ and take $\tilde{c} = -(A-\sigma I)^{-1}b$. Then, $\tilde x$ defined by \eqref{eq:xtilde} satisfies the ODE
\begin{align}
\dot{\tilde{x}} = (A - \sigma I) \tilde{x} + \tilde{b} \rho + \tilde{f}(\tilde{x},\rho) \label{eq:gov2}
\end{align}
for some $\tilde{b} \in \mathbb{R}^n$ and  $\tilde{f} \in C^{\omega}(\mathbb{R}^{n+1}, \mathbb{R}^n)$ with $\tilde{f}(y) = O(|y|^2)$ as $y \to 0$. Furthermore, if \eqref{eq:gov2} has a one-dimensional smooth manifold $\tilde{W}^{\rho}(0)$ satisfying \eqref{eq:mani_anal} then \eqref{eq:gov_n} has a one-dimensional smooth manifold $W^{\rho}(0)$ satisfying \eqref{eq:mani_anal}. \label{lem:trans}
\end{lem}

\begin{proof} We start with deriving \eqref{eq:gov2}. Using \eqref{eq:gov_n} we compute
\begin{gather}
\begin{aligned}
  \dot{\tilde x}
  &= \dot{x}/ \rho - \sigma x / \rho,  \\
  &= (A - \sigma I) x/ \rho + b + f(x,\rho)/\rho, \\
  &= (A - \sigma I) (\tilde x + \tilde c) + b + f(\rho(\tilde{x}+\tilde{c}), \rho)/\rho, \\
  &= (A - \sigma I) \tilde x + f(\rho (\tilde{x}+\tilde{c}), \rho)/\rho.
\end{aligned} \label{eq:dxtilde}
\end{gather}
Since $f(y) \in O(|y|^2)$ is analytic, we can write it as $f(y) = B(y,y) + \hat f(y)$ for some $\hat f(y) \in O(|y|^3)$ and some bilinear map $B : (\R^{n+1})^2 \to \R^n$. Then, \[ 
  f(\rho (\tilde{x}+\tilde{c})  ,  \rho) / \rho
  = \rho B( (\tilde{x}+\tilde{c}, 1), (\tilde{x}+\tilde{c}, 1) ) +  \hat{f}(\rho (\tilde{x}+\tilde{c})  ,  \rho)/\rho.
\]
Thus, taking $\tilde b := B( (\tilde{c}, 1), (\tilde{c}, 1) )$ and 
\[
  \tilde f(\tilde x,\rho) := \hat{f}(\rho (\tilde x+\tilde{c})  , \rho)/\rho + \rho B \big( ( \tilde{x},0) + 2(\tilde c, 1), (\tilde{x}, 0) \big),
\]  
Equation \eqref{eq:gov2} follows and $\tilde f(\tilde y) \in O(|\tilde y|^2)$ because $\tilde f (\tilde x,\rho) = O(\rho^2 + |\rho \tilde x|)$ as $z \to 0$. 

It remains to show that if \eqref{eq:gov2} has a one-dimensional smooth manifold $\tilde{W}^{\rho}(0)$  satisfying \eqref{eq:mani_anal}, then \eqref{eq:gov_n} has a one-dimensional smooth manifold $W^{\rho}(0)$ satisfying \eqref{eq:mani_anal}. 
Using that \eqref{eq:gov2} has a one-dimensional smooth manifold  $\tilde{W}^{\rho}(0)$ satisfying \eqref{eq:mani_anal}, we infer that the trajectories of the $\rho$-analytic solutions $(\tilde{x},\rho)$ for $\rho >0, \rho=0$ or $ \rho<0$ analytically connect at $0$. Given any of these three solutions, it is easy to verify that $({x},\rho)$ given by \eqref{eq:xtilde} is a $\rho$-analytic solution of \eqref{eq:gov_n}. We claim that it is the unique $\rho$-analytic solution of \eqref{eq:gov_n} modulo translation in time when restricted to $\rho >0, \rho=0$ or $ \rho<0$. Relying on this claim, it follows from \eqref{eq:xtilde} that \eqref{eq:gov_n} has  a one-dimensional smooth manifold $W^{\rho}(0)$ satisfying \eqref{eq:mani_anal}.

We are left with proving the claim. The case $\rho = 0$ is trivial and the cases $\rho > 0$ and $\rho < 0$ can be treated similarly. We focus on the case $\rho > 0$. Let $(\hat x, \hat \rho)$ be any $\hat \rho$-analytic solution of \eqref{eq:gov_n}. By translating time we may assume that $\hat \rho = \rho$. Since $(\hat x, \rho)$ is $\rho$-analytic, there exist an $R > 0$ and a local, analytic function $\hat g$ with $\hat g(0) = 0$ such that $\hat x = \hat g (\rho)$ on $(-\infty, -R)$. Then, there exists a unique coefficient $\hat c \in \R^n$ such that $(\overline x, \rho)$ with $\overline x := \hat x / \rho - \hat c$ is $\rho$-analytic. Similar to \eqref{eq:dxtilde}, we obtain that $\overline x$ satisfies
\begin{equation} \label{pfz}
  \dot{\overline x}
  = (A - \sigma I) \overline{x} + (A - \sigma I) (\hat c - \tilde c) +  \overline{b} \rho + \overline{f}(\overline{x},\rho) 
\end{equation}
for some $\overline{b} \in \mathbb{R}^n$ and  $\overline{f} \in C^{\omega}(\mathbb{R}^{n+1}, \mathbb{R}^n)$ with $\overline{f}(y) = O(|y|^2)$ as $y \to 0$. Since $\rho$-analyticity implies that $\dot{\overline x}(t) \to 0$ as $t \to -\infty$, it follows that the right-hand side in \eqref{pfz} vanishes as $(\overline{x}, \rho) \to 0$. Hence, $\hat c = \tilde c$, and thus $(\overline{x}, \rho)$ satisfies the same ODE as $(\tilde {x}, \rho)$. By the uniqueness of $\tilde x$ it follows that $\overline x$ is unique, and therefore that $\hat x$ is unique.
\end{proof}

\begin{proof}[Proof of Theorem \ref{theo:main}] Let $\Lambda := \max_{1 \leq i \leq n} {\rm Re}(\lambda_i)$ and set $K = \lfloor \Lambda/\sigma \rfloor + 1$. Applying Lemma \ref{lem:trans} $K$-times recursively (relying on $\rm{Re} (\lambda_i) \notin \sigma \N_+$) we obtain that the eigenvalues of the resulting matrix in \eqref{eq:gov2} are in the negative half-plane. Consequently, we can apply Lemma \ref{lem:negeig}. Then, it follows by preservation of $W^{\rho}(0)$ from Lemma \ref{lem:trans} that Equation \eqref{eq:gov_n} has a one-dimensional smooth manifold $W^{\rho}(0)$ satisfying \eqref{eq:mani_anal}.  
\end{proof}

\section{Application \label{sec:application}}

We first reformulate Theorem \ref{theo:main} in the setting of the singular ODE \eqref{eq:intro_an_n}--\eqref{eq:intro_n} and then apply it to the Ballistic Ageing Thin viscous Sheet (BATS) model for fungal tip growth \cite{jong2020modelling}.

%

\subsection{Uniqueness of solutions of Singular ODEs}

Recall \eqref{eq:intro_an_n}, \eqref{eq:intro_n}, \eqref{eq:V}. We obtain the following corollary from Theorem \ref{theo:main}.

\begin{cor} Let $\lambda_1, \cdots, \lambda_n$ be the eigenvalues of $A$. If ${\rm Re}(\lambda_i) \notin 2 \N_+$ for all $1 \leq i \leq n$, then there exists a unique solution $x$ of \eqref{eq:intro_n} satisfying \eqref{eq:intro_an_n}.
\label{cor:main}
\end{cor}

\begin{proof}  
Take $\sigma =2$ in \eqref{eq:gov_n}. Theorem \ref{theo:main} implies that there exists a unique solution $(x, \rho)$ to \eqref{eq:gov_n} up to translation in time with $\rho > 0$ and $x(t) = g(\rho(t))$ for some $g \in C^{\omega}((-\varepsilon,\varepsilon), \R^n)$ and all $t$ small enough. Since the transformation introduced in Section \ref{sec:intro} to transform \eqref{eq:intro_an_n}--\eqref{eq:intro_n} into \eqref{eq:gov_n} is invertible, Corollary \ref{cor:main} follows. 
\end{proof}

\subsection{The BATS model \label{sec:bats}}

The BATS model \cite{jong2020modelling} describes the shape of a single axially symmetric fungal cell wall during growth. It assumes a constant speed of growth and an equilibrium shape of the cell tip in the co-ordinate frame which moves along the cell tip. The independent variable describing the cell wall is the arclength $s$. Specifically, we have that $s \rightarrow 0$ describes the tip of the cell, see Figure \ref{fig:cellshape}. 

In \cite{jong2019numerics} the shape of the cell tip is computed numerically with asymptotic expansions. The authors observed that for certain special choices of the parameters in the BATS model the coefficients in these expansions blow up and the expansions fail to capture the shape of the cell tip. Yet, no theoretical explanation was found for this observation. 

Our aim is to seek such theoretical explanation. We will cast the system of ODEs of the BATS model in a form to which Corollary \ref{cor:main} can be applied. The non-resonance condition on the eigenvalues translates to a condition on the parameters of the BATS model. If the BATS model has a unique, local, analytic solution, then we expected that the numerically computed solutions constructed from asymptotic expansions converge to the corresponding coefficients in the power series of the exact solution. Furthermore, it will turn out that the non-resonance condition in Corollary \ref{cor:main} precisely characterizes all cases in \cite{jong2019numerics} where the coefficients blow up. This demonstrates in a specific setting the necessity of the non-resonance condition in Corollary \ref{cor:main} and Theorem \ref{theo:main}. 
\smallskip

\begin{figure}[h]
\centering
\includegraphics[width=5cm]{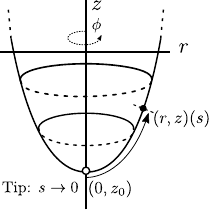}
\caption{Fungal tip growth shape: The surface is parametrised with respect to the arclength $s$ to the cell's tip and azimuth variable $\phi$. In terms of the $(r,z)$-variables the tip is located at $(0,z_0)$. \label{fig:cellshape} }
\end{figure}   

First, we introduce the BATS model. We consider the phase space given by 
\begin{equation}
M= \{ (\varsigma , h , \Psi, z,r) \in  (-1,1) \times \mathbb{R} \times \mathbb{R} \times \mathbb{R}  \times \mathbb{R}_+     \}.
\end{equation}
The $h$ variable represents the cell wall thickness, $\Psi$ is the age of the cell wall material, $z$ is the axial co-ordinate variable, $r$ is the radial distance variable and $\varsigma = dr/ds$. We note that $\varsigma \in (-1,1)$ since parametrization of $z,r$ by $s$ gives the equality $(dr/ds)^2 + (dz/ds)^2=1$. The governing equations are given by \cite{jong2020modelling}:
\begin{gather}\left\{
\begin{aligned}
\frac{d\varsigma}{ds} &= \frac{3}{2} \frac{1-\varsigma^2}{r} \left(-1+  \frac{\Gamma(z,r) \mu(\Psi) \varsigma \sqrt{1-\varsigma^2}}{r^3 }  \right)  \\
 \frac{dh}{ds}& = \left( \frac{r \gamma(\varsigma,z,r)}{\Gamma(z,r)}- \frac{\varsigma}{2r}- \frac{r^2 }{2 \Gamma(z,r) \mu(\Psi) \sqrt{1-\varsigma^2}}\right) h  \\
\frac{d \Psi}{ds} &= \frac{rh}{\Gamma(z,r)} - \frac{r \gamma(\varsigma,z,r) }{\Gamma(z,r) } \Psi\\
 \frac{d z}{ds} &= \sqrt{1- \varsigma^2} \\ 
 \frac{dr}{ds} &= \varsigma ,
\end{aligned} \right. \label{eq:finalfull}
 \end{gather}
where 
\begin{align*}
\gamma(\varsigma,z,r)=
  \frac{r  \sqrt{1-\varsigma^2}-z  \varsigma}{(z^2+r^2)^{3/2}},
 \qquad \Gamma(z,r) = 1 + \frac{z}{\sqrt{r^2+z^2}} ,
\end{align*}
and $\mu  \in C^\omega(\mathbb{R}_+, \mathbb{R}_+)$  satisfies
\begin{align}
\mu \text{ is increasing}
\quad \text{and} \quad
\lim_{\Psi \rightarrow \infty}\mu(\Psi)= \infty.
\end{align}
The function $\mu$ corresponds to viscosity. The viscosity of the cell wall increases with age which corresponds to hardening of the cell wall. 
\\

For the tip shape to be smooth we require two conditions on $(\varsigma, h,\Psi,z,r)$ as $s \to 0$:

\begin{itemize}
\item[T1] \textbf{Tip limits:} there exist $h_0 >0$ and $z_0<0$ such that  
\begin{gather}
\lim_{s \rightarrow 0} (\varsigma,h,\Psi,z,r)(s) 
= (1, h_0, h_0 z_0^2, z_0, 0); \label{eq:limits}
\end{gather}
\item[T2] \textbf{Analyticity:} there exist $s_1>0$ and $g \in C^{\omega}\left( (-\varepsilon ,\varepsilon), \mathbb{R}^4 \right)$ with $\varepsilon=r(s_1)^2$ such that 
\begin{align}
(\varsigma,h,\Psi,z)(s)= g(r(s)^2) \qquad \forall s \in (0,s_1). \label{eq:anal}
\end{align}
\end{itemize}
Condition T1 follows from local analysis of solutions with a tip \cite{jong2020modelling}. Specifically, $z_0$ corresponds to the distance of the tip to the cell wall producing organelle. Condition T2 is a result of requiring a smooth shape at the tip as in Figure \ref{fig:cellshape}. 
It allows for expressing the solutions as $(\varsigma, h,\Psi,z)$ as an even analytic function of $r$ on a neighborhood around $0$.
\smallskip

Next we write the BATS model and its desired solution in the form \eqref{eq:intro_an_n}--\eqref{eq:intro_n}. Since $z$ and $r$ are dependent variables, we can reduce the number of equations from five to four. We do so by considering $r$ as the variable which replaces $s$. Simultaneously, we change the unknown $\varsigma$ to 
\[
  \eta = \frac{\sqrt{1-{ \varsigma}^2}}{r} \qquad \big(\text{note that } \varsigma = \sqrt{1-\eta^2 r^2} \big) 
\]
to avoid a removable singularity. 
Denoting by $'$ the derivative with respect to $r$, we obtain
\begin{gather}
\left\{\begin{aligned}
\eta '  &=  \frac{\eta}{2 r} \left(1 -3  \frac{\mu({\Psi}) \eta \sqrt{1-\eta^2 r^2}}{\Xi({z},r^2)} \right) ,\\
h'  &= \left(  \frac{ \xi(\eta, {z}, r^2) {\Xi}({z},r^2)}{ \sqrt{1-\eta^2 r^2}} -  \frac{ {\Xi}({z},r^2)   }{ 2 {\mu}({\Psi})   \eta \sqrt{1-\eta^2 r^2}}  - \frac{1}{2   }  \right) \frac{h}{r} , \\
\Psi' &=  {\Xi}({z},r^2) \frac{ {{h   -  \xi(\eta, {z},r^2) \Psi   }}}{ r \sqrt{1-\eta^2 r^2}}, \\
z ' &= \frac{r \eta }{ \sqrt{1-\eta^2 r^2}},\\
\end{aligned}\right.  \label{eq:odetfi}
\end{gather}
with 
\begin{align*}
\xi(\eta, {z}, r^2 ) 
&:= \gamma \Big(\frac{\sqrt{1-{ \varsigma}^2}}{r}, {z}, r \Big) 
= \frac{ r^2 \eta - {z} \sqrt{1-\eta^2 r^2}}{(r^2+{z}^2)^{3/2}}, \\ \Xi({z},r^2) 
&:= \frac{r^2}{\Gamma(z,r)}
= r^2+{z}^2 - {z} \sqrt{r^2+{z}^2}.
\end{align*}

Next we compute the initial condition. From \eqref{eq:limits} this is trivial for $h, \Psi, z$. The initial condition for $\eta$ requires some computation. Indeed, while
\[
  \lim_{r \to 0} \eta (r) = \lim_{s \to 0} \frac{z'(s)}{r(s)},
\]
the enumerator and denominator vanish as $s \to 0$. Using l'Hopital and noting from $z' = \sqrt{1 - \varsigma^2} = \sqrt{1 - (r')^2}$ that $z'' = -r'' r' / z' = - \varsigma' r' / z'$ we obtain
\begin{align*}
  \lim_{r \to 0} \eta (r) &= \lim_{s \to 0} \frac{z''(s)}{r'(s)} 
  = \lim_{s \to 0} -\frac{\varsigma'(s)}{z'(s)} \notag \\ \nonumber
  &= \lim_{s \to 0}  \frac{3}{2} \frac{\sqrt{1-\varsigma(s)^2}}{r(s)} \left( 1 - \frac{\Gamma(z(s),r(s)) }{r(s)^2 } \frac{\sqrt{1-\varsigma(s)^2}}{r(s) } \mu(\Psi(s)) \varsigma(s) \right) \\ \nonumber
  &= \frac{3}{2} \lim_{r \to 0} \eta(r) \left(  1 - \frac{ \eta(r) }{ \Xi(z(r),r^2) } \mu(\Psi(r)) \sqrt{1- \eta(r)^2 r^4} \right) \\ \nonumber
   &= \frac{3}{2} \lim_{r \to 0} \eta(r) \left( 1 - \frac{\eta(r)}{2 z_0^2  } \mu(h_0 z_0^2) \right). 
\end{align*}
Solving for $\eta(0)$ yields $\eta(0) = 0$ or $\eta(0) = 2 z_0^2 / (3 \mu(h_0 z_0^2))$. If $\eta(0) = 0$, then \eqref{eq:odetfi} and T1 imply $\lim_{r \to 0}r h'(r) = - \infty$, which contradicts with T2. Therefore, we only consider $\eta(0) = 2 z_0^2 / (3 \mu(h_0 z_0^2))$. In conclusion, we obtain
\begin{gather}
\lim_{r \rightarrow 0} (\eta,h,\Psi,z)(r) 
= \left( \frac{2 z_0^2}{3 \mu(h_0 z_0^2)}      , h_0,   h_0z_0^2, z_0 \right)
=: p_0(h_0,z_0). \label{eq:limits:r}
\end{gather}

Finally, T2 implies directly local analyticity of $h, \Psi, z$ and ensures that the odd coefficients of the expansions for $h, \Psi, z$ are zero. Then, writing $\eta(r) = \frac{ z'(r) \varsigma(r)}{r}$ 
and using that all even coefficients of $z'(r)$ are zero, we conclude that $\eta$ satisfies the same property. In conclusion, there exist $\varepsilon>0$ and $\tilde g \in C^{\omega}\left( (-\varepsilon ,\varepsilon), \mathbb{R}^4 \right)$ such that 
\begin{align}
(\eta,h,\Psi,z)(r)= \tilde g(r^2) \qquad \forall r \in (0, \varepsilon). \label{eq:anal:r}
\end{align}
 
Next we cast \eqref{eq:odetfi} in the form \eqref{eq:intro_n}. Suppose there exists a local solution $\overline x := ( \eta,  h ,  \Psi,  z)$ to \eqref{eq:odetfi}. Let $r_0$ be small enough such that $\overline x$ exists on $[0, r_0]$, $\sup_{(0, r_0)} z < 0$ and $\inf_{(0, r_0)} \min\{ \Psi, \eta \} > 0$. Then, the right-hand side in \eqref{eq:odetfi} can be written as $\overline V (\overline x, r^2)/r$, where $\overline V : \R^5 \to \R^4$ is analytic in a neighborhood around $p_0(h_0,z_0)$. To obtain the initial condition in \eqref{eq:intro_n}, we shift variables to ${x}  := \overline{x} -  p_0(h_0,z_0)$ and set $V(x, r^2) := \overline V (x + p_0, r^2)$. We observe from \eqref{eq:anal:r} that $x$ satisfies \eqref{eq:intro_an_n}.

Finally, we note that this transformation can easily be inverted, i.e.\ if \eqref{eq:odetfi}--\eqref{eq:limits:r} has a solution satisfying \eqref{eq:anal:r}, then \eqref{eq:finalfull}--\eqref{eq:limits} has a solution satisfying T2. Indeed, $\varsigma(r) = \sqrt{1- \eta(r)^2 r^2}$ satisfies the condition in T2. Introducing $s(r)$ as the solution of $\frac{ds}{dr}(r) = 1/\varsigma(r), s(0) = 0 $, we apply the inverse function theorem (relying on $\frac{ds}{dr}(0) =1 \neq 0$) to parametrize $\varsigma, h, \Psi, z, h$ in $s$ around $s=0$. Then, \eqref{eq:finalfull}--\eqref{eq:limits} follows.
\smallskip

To summarize the above, \eqref{eq:finalfull}--\eqref{eq:limits} has a solution satisfying T2 if and only if \eqref{eq:intro_n} has a solution satisfying \eqref{eq:intro_an_n}, where $V$ is as constructed above. Hence, we may work with \eqref{eq:intro_an_n}--\eqref{eq:intro_n} in the remainder. Corollary \ref{cor:main} provides a sufficient condition for the existence and uniqueness of solutions to \eqref{eq:intro_n} which satisfy \eqref{eq:intro_an_n}. To make this condition explicit, we need to compute the eigenvalues of $A := \nabla_x V(0)$ (see \eqref{eq:V}). From \eqref{eq:odetfi} we compute
\begin{align}
A = \nabla_x V(0) =
\begin{bmatrix}
 -\frac{1}{2} & 0 & -\frac{z_0^2 \mu '\left(h_0 z_0^2\right)}{3 \mu \left(h_0 z_0^2\right){}^2} & \frac{2 z_0}{3 \mu \left(h_0 z_0^2\right)} \\
 \frac{9 h_0 \mu \left(h_0 z_0^2\right)}{4 z_0^2} & 0 & \frac{3 h_0 \mu '\left(h_0 z_0^2\right)}{2 \mu \left(h_0 z_0^2\right)} & -\frac{3 h_0}{z_0}  \\
 0 & 2 z_0^2 & -2 & 4 h_0 z_0 \\
 0 & 0 & 0 & 0  
\end{bmatrix} . \label{eq:A}
\end{align}
The eigenvalues corresponding to $A$ are given by 
\begin{gather}
\begin{aligned}
\lambda_1(h_0,z_0) &=0 , \; \lambda_2(h_0,z_0) = 0 ,\\ 
  \lambda_3(h_0,z_0) &= \frac{1}{4} \left(-5 -\sqrt{\frac{48 h_0 z_0^2 \mu '\left(h_0 z_0^2\right)}{\mu \left(h_0 z_0^2\right)}+9}\right),\\
     \lambda_4(h_0,z_0) & =  \frac{1}{4} \left( -5 + \sqrt{\frac{48 h_0 z_0^2 \mu '\left(h_0 z_0^2\right)}{\mu \left(h_0 z_0^2\right)}+9}\right).
\end{aligned} \label{eq:Alambda}
\end{gather} 
Since $\lambda_3(h_0,z_0) < 0$ for all $h_0 > 0 > z_0$, the condition in Corollary \ref{cor:main} translates to 
\begin{align} \label{bats:nonres:cond}  
  \frac12 \lambda_4(h_0,z_0) = \frac{1}{8} \left( -5 + \sqrt{\frac{48 h_0 z_0^2 \mu '\left(h_0 z_0^2\right)}{\mu \left(h_0 z_0^2\right)}+9}\right)
  \notin \N.
\end{align}

In conclusion, \eqref{bats:nonres:cond} gives a sufficient conditions on the parameters $h_0 > 0 > z_0$ of the BATS model (see \eqref{eq:limits}) under which the BATS model describes a unique shape for the cell tip in the class of local analytic functions. Since this condition is a new result, we compare it with the findings in \cite{jong2019numerics} mentioned at the start of Section \ref{sec:bats}. In Appendix \ref{app:bats} we show that there is a one-to-one connection between \eqref{bats:nonres:cond} and the values of $h_0, z_0$ for which the asymptotic expansions for the solution in \cite{jong2019numerics} fail. This demonstrates that non-resonance conditions are indeed required in practice, and that the condition in Corollary \ref{cor:main} is in fact minimal at least in the particular case of the BATS model investigated in \cite{jong2019numerics}.

%

\section{Concluding remarks and future work \label{sec:conc}}

Corollary \ref{cor:main} provides a new tool for obtaining existence and uniqueness for solutions in the sense of \eqref{eq:intro_an_n} to singular ODEs of type \eqref{eq:intro_n}. Such ODEs appear for instance in models for the shape of axially symmetric surfaces. In Section \ref{sec:bats} we have demonstrated that Corollary \ref{cor:main} provides new properties for the BATS model, and that the sufficient conditions in Corollary \ref{cor:main} can be minimal in practice. The tangent space at the equilibrium uniquely determines the one dimensional unstable manifold in Lemma \ref{lem:negeig}. Then, it follows from Lemma \ref{lem:trans} that an expansion of sufficiently high order approximates the desired analytic solution.
\smallskip

Our results open up four interesting problems. First, we expect that Theorem \ref{theo:main} also applies if in the equation for $\rho$ in \eqref{eq:gov_n} a nonlinear term is included. Indeed, this does not alter the linearized equation and thus the proof of Lemma \ref{lem:negeig} will remain identical. While Lemma \ref{lem:trans} requires modifications since additional nonlinear terms appear when transforming the ODE, these terms can be absorbed in the nonlinearities corresponding to the $x$-component. We left out this generalization because the applications which we have in mind are captured by the linear setting.

Second, from a proof perspective we expect that a more direct approach would work which only relies on a contraction-type argument. Specifically, we could rewrite \eqref{eq:gov_n} as the non-autonomous ODE as in \eqref{eq:nonaut}. For \eqref{eq:nonaut} we can write a Duhamel-formula and proceed with an application of Banach's fixed point theorem. Such a proof would be somewhat technical. In our approach the technicalities of a fixed point argument are hidden in the application of the unstable manifold theorem. We note that an approach by applying Poincar\'{e}-Dulac to \eqref{eq:gov_n} does not work if we only assume that $\lambda_i \notin \sigma \N_+$, because there might exists a $\lambda_i = [\lambda, \sigma ] \cdot k $ with $k \in \N^{n+1}_+$, $|k| \geq 2$ and $k_{n+1} \geq 1$ \cite{anosov1997ordinary, broer2009normal}. Furthermore, Poincar\'{e}-Dulac would only yield a formal transformation which is not necessarily analytic. 

Third, one can try to generalize Theorem \ref{theo:main} to system \eqref{eq:gov_n} in which the nonlinear term $f$ is merely Lipschitz. Then, our argument by applying Lemma \ref{lem:trans} recursively does not work. Instead, it seems that one is forced to apply a Banach's fixed point type argument.

Fourth, in the setting of the BATS application in Section \ref{sec:bats} it is desirable to know whether solutions depend continuously on the parameters $h_0,z_0$. This translates to the question on whether solutions to \eqref{eq:intro_n} of type \eqref{eq:intro_an_n} are continuous with respect to perturbations of $V$. To answer this question, the procedure of Section \ref{sec:proof} can be repeated. However, in addition it needs to be shown that $W^{\rho}(0)$ and the coefficients obtained in Lemma \ref{lem:trans} are continuous with respect to the perturbation of $V$. We expect that center manifold theory \cite{carr2012applications} may provide tools to prove this.

\appendix 

\section{The BATS model for specific $\mu$ \label{app:bats}}

In \cite{jong2019numerics} the BATS model from Section \ref{sec:bats} is considered for the following choices of the viscosity function: 
\begin{align}
\mu_m(\Psi) = 1+ \Psi^m \quad \text{for } m=2,3,4,5. \label{eq:visc_m}
\end{align}
They construct expansions for the solutions to the BATS model. They observed that for $m=2,3$ the coefficients in their expansions were well-defined for any choice of the parameters $h_0 > 0 > z_0$, but that for $m =4,5$ the coefficients were singular if and only if 
\begin{align} \label{app:h0z0}
\begin{aligned}
  (h_0 z_0^2)^4 &= 5 &&\text{for } m = 4 \\
  (h_0 z_0^2)^5 &= 2 &&\text{for } m = 5.
\end{aligned}
\end{align}

Here we investigate to which extend these observations match with the condition \eqref{bats:nonres:cond}. From \eqref{eq:Alambda} we observe that  
\begin{align*}
\lambda_4(h_0,z_0;\mu_m) 
= \frac{1}{4} \left( -5 + \sqrt{\frac{48 m}{1 + (h_0 z_0^2)^{-m}}+9}\right)
\leq   \frac{1}{4} \left( -5 + \sqrt{ 48 m +9}\right),
\end{align*} 
where we have added the dependence of $\mu_m$ in the arguments of $\lambda_4$.
In particular, 
\begin{align*}
\lambda_4(h_0,z_0;\mu_m) < \begin{cases}
  2
  &\text{if } m = 2,3  \\
  3
  &\text{if } m = 4,5.
\end{cases}
\end{align*}
Consequently, for $\mu_2,\mu_3$ the condition in \eqref{bats:nonres:cond} imposes no restrictions on $(h_0,z_0)$, and for $\mu_4,\mu_5$ this condition
translates to $\lambda_4(h_0,z_0; \mu_m) \neq 2$. Since $\lambda_4$ is increasing in $h_0 z_0^2$ (see the display above), this corresponds to a single value for $h_0 z_0^2$, and it is readily verified that this value is given by \eqref{app:h0z0}. Hence, for each case examined in  \cite{jong2019numerics}, condition \eqref{bats:nonres:cond} characterizes precisely those values of $h_0, z_0$ for which the expansions in \cite{jong2019numerics} are not singular.

\bibliographystyle{alpha}
\bibliography{my_bib}{}

\end{document}